\documentclass[12pt,leqno]{amsart}
\usepackage{amssymb,amsthm,amsmath,latexsym}
\usepackage[all]{xy}

\newtheorem{theorem}{\sc Theorem}[section]

\newtheorem{prop}[theorem]{\sc Proposition}
\newtheorem{cor}[theorem]{\sc Corollary}

\newtheorem{rem}[theorem]{\sc Remark}
\newtheorem{ex}[theorem]{\sc Example}

\DeclareMathOperator{\tensor}{T_{\otimes}}
\newtheorem*{thmA}{Theorem A}
\newtheorem*{thmB}{Theorem B}
\newtheorem*{thmC}{Theorem C}

\title[Non-abelian tensor product]{Finiteness of homotopy groups related to the non-abelian tensor product
}
\author[Bastos]{Raimundo Bastos}
\address{ Departamento de Matem\'atica, Universidade de Bras\'ilia,
Brasilia-DF, 70910-900 Brazil }
\email{(Bastos) bastos@mat.unb.br;  (Rocco) norai@unb.br.}
\author[Rocco]{Nora\'i R. Rocco }
\author[Vieira]{Ewerton R. Vieira}
\address{ Departamento de Matem\'atica, Universidade Federal de Goi\'as,
Goi\^ania-GO, 74690-900 Brazil }
\email{(Vieira) ewerton@ufg.br}
\thanks{This work was partially supported by FAPDF - Brazil, Grant: 0193.001344/2016.}

\begin{document}
\maketitle
\begin{abstract}
By using finiteness related result of non-abelian tensor product we prove finiteness conditions for the homotopy groups $\pi_n(X)$ in terms of the number of tensors. In particular, we establish a quantitative version of the classical Blakers-Massey triad connectivity theorem. Moreover, we study others finiteness conditions and equivalence properties that arise from the non-abelian tensor square. In the end, we give applications to homotopy pushout, especially in the case of Eilenberg-MacLane spaces.
\end{abstract}
\maketitle
\section{Introduction}

In \cite{BL} Brown and Loday presented a topological significance for the non-abelian tensor product of groups. The non-abelian tensor product is used to describe the third relative homotopy group of a triad as a {\it non-abelian tensor product} of the second homotopy groups of appropriate subspaces. More specifically,  in \cite[Corollary 3.2]{BL}, the third triad homotopy group is 
$$
\pi_3(X,A,B) \cong  \pi_2(A,C)\otimes \pi_2(B,C), 
$$
where $X$ is a pointed space and $\{A, B\}$ is an open cover of $X$ such that $A, B$ and $C=A\cap B$ are connected and $(A, C), (B, C)$ are $1$-connected. 

In \cite{Ellis}, Ellis gave a finiteness criterion for the triad homotopy group in terms of the finiteness of the involved groups (see also \cite{BNRa,BNRb}). More generally, finiteness conditions to $\pi_n(X)$ when excision theorem holds is given by the finiteness of $\pi_n(A)$, $\pi_n(B)$ and $\pi_{n-1}(C)$. However, when the excision property does not hold, the failure is mesured by the triad homotopy groups $\pi_n(X,A,B)$ with $n\geq 3$. Therefore, finiteness of $\pi_n(X)$ also depends on the triad homotopy groups $\pi_n(X,A,B)$. Thus, in order to give a bound for $\pi_n(X)$, it is needed to study finiteness conditions on $\pi_n(X,A,B)$. In \cite{BL2}, another application of the non-abelian tensor product is done by Brown and Loday, where they extended the classical Blakers-Massey triad connectivity theorem, which states that if $A,B,$ and $A\cap B$ are connected, $\{A,B\}$ is an open cover of $X$, $(A,A\cap B)$ is p-connected, and $(B,A\cap B)$ is q-connected, then $\pi_{p+q+1}(X,A,B)$ is isomorphic to the non-abelian tensor product $$\pi_{p+1}(A,A\cap B)\otimes \pi_{q+1}(B,A\cap B).$$ 
The hypothesis $p,q\geq 2$ is broadened to $p,q\geq 1$, and the hypothesis $\pi_1(A\cap B)=0$ is removed. 

For the convenience of the reader we repeat the relevant definitions (cf. \cite{BNRa,BNRb,Nak}). Let $G$ and $H$ be groups each of which acts upon the other (on the right),
\[
G\times H \rightarrow G, \; (g,h) \mapsto g^h; \; \; H\times G \rightarrow
H, \; (h,g) \mapsto h^g
\]
and on itself by conjugation, in such a way that for all $g,g_1 \in G$ and
$h,h_1 \in H$,
\begin{equation*}   \label{eq:0}
g^{\left( h^{g_1} \right) } = \left( \left( g^{g^{-1}_1}  \right) ^h \right) ^{g_1} \; \; \mbox{and} \; \; h^{\left( g^{h_1}\right) } =
\left( \left( h^{h_1^{-1}} \right) ^g \right) ^{h_1}.
\end{equation*}
In this situation we say that $G$ and $H$ act {\em compatibly} on each other. Let  $H^{\varphi}$ be
an extra copy of $H$, isomorphic via $\varphi : H \rightarrow
H^{\varphi}, \; h \mapsto h^{\varphi}$, for all $h\in H$.
Consider the group $\eta(G,H)$ defined in  \cite{Nak} as
$$\begin{array}{ll} {\eta}(G,H) =  \langle
G \cup H^{\varphi}\ |  &
[g,{h}^{\varphi}]^{g_1}=[{g}^{g_1},({h}^{g_1})^{\varphi}], \;
[g,{h}^{\varphi}]^{h^{\varphi}_1} = [{g}^{h_1},
({h}^{h_1})^{\varphi}] , \\ & \ \forall g,g_1 \in G, \; h, h_1 \in H
\rangle . \end{array}$$

It is a well known fact (see \cite[Proposition 2.2]{Nak}) that the subgroup
$[G, H^{\varphi}]$ of $\eta(G,H)$ is canonically isomorphic with the {\em non-abelian
tensor product} $G \otimes H$, as defined by Brown and Loday in their seminal paper \cite{BL}, the isomorphism being induced by $g \otimes h \mapsto
[g, h^{\varphi}]$ (see also Ellis and Leonard \cite{EL}). It is clear that the subgroup $[G,H^{\varphi}]$ is normal in $\eta(G,H)$ and one has the decomposition
\begin{equation*} \label{eq:decomposition}
 \eta(G,H) = \left ( [G, H^{\varphi}] \cdot G \right ) \cdot H^{\varphi},
\end{equation*}
where the dots mean (internal) semidirect products. We observe that when $G=H$ and all actions are conjugations, $\eta (G,H)$ becomes the group $\nu (G)$ introduced in \cite{NR1}. Recall that an element $\alpha \in \eta(G,H)$ is called a {\em tensor} if $\alpha = [a,b^{\varphi}]$ for suitable $a\in G$ and $b\in H$.  We write $\tensor(G, H)$ to denote the set of all tensors (in $\eta(G,H)$). When $G = H$ and all actions are by conjugation, we simply write $\tensor(G)$ instead of $\tensor(G,G)$. A number of structural results for the non-abelian tensor product of groups (and related constructions) in terms of the set of tensors where presented in \cite{BNRa,BNRb,BR1,BR2,NR2}).  

Our contribution is to give finiteness conditions and bounds to the triad homotopy groups $\pi_n(X,A,B)$, furthermore finiteness condition and bound is given to $\pi_n(X)$.  We establish the following related results.

\begin{thmA}
Let $X$ be a union of open subspaces $A$, $B$ such that $A, B$ and $C=A\cap B$ are path-connected, and the pairs $(A,C)$ and $(B,C)$ are respectively $p$-connected and $q$-connected. Suppose that $\pi_n(A)$, $\pi_n(B)$, $\pi_{n-1}(C)$ and the set of tensors $\tensor(\pi_{p+1}(A,C),\pi_{q+1}(B,C))$ are finite, where $n=p+q+1$. Then $\pi_{n}(X)$ is a finite group with $\{a, b, c, m\}$-bounded order, where $|\pi_n(A)|=a$, $|\pi_n(B)|=b$, $|\pi_{n-1}(C)|=c$ and $m = |\tensor(\pi_{p+1}(A,C),\pi_{q+1}(B,C))|$.
\end{thmA}

An application of Theorem A is that
$\pi_3( K(C_{r^\infty},2)\vee K(C_{s^\infty},2))$ is trivial, where $``\vee "$ is the wedge sum,  $r$ and $s$ are primes, and $K(G,n)$ is an  {\it Eilenberg-MacLane space} (i.e., a topological space having just one nontrivial homotopy group $\pi_n(K(G,n))\cong G$). See also Remark \ref{rem_thmA}  and Corollary \ref{cor.bounded} below. 

In \cite{BL}, Brown and Loday shown that the third homotopy group of the suspension of an Eilenberg-MacLane space $K(G,1)$ satisfies
$$
\pi_3(SK(G,1)) \cong J_2(G),
$$
where $J_2(G)$ denotes the kernel of the derived map $\kappa: [G,G^{\varphi}] \to G'$, given by  $[g,h^{\varphi}] \mapsto [g,h]$ (cf. \cite[Chapter 2 and 3]{NR2}). Many authors had studied bounds to the order of $\pi_3(SK(G,1))$ (cf. \cite{ADPT,BK,BJR,NR14}). Now, we can deduce a finiteness criterion for $\pi_3(SX)$ in terms of $\pi_2(X)$ and the number of tensors $\tensor(G)$, where $\pi_1(X) \cong G$ and $SX$ is the suspension of a space $X$ (see Remark \ref{rem.suspension}, below).

\begin{thmB}
Let $X$ be a connected space and $\pi_1(X)=G$. Suppose that the set of tensors $\tensor(G)$ has exactly $m$ tensors in $\nu(G)$ and $\pi_2(X)$ is finite with $|\pi_2(X)| = a$. Then $\pi_3(SX)$ is a finite group with $\{a, m\}$-bounded order.
\end{thmB}

It is well known that the finiteness of the non-abelian tensor square $[G,G^{\varphi}]$ does not imply the finiteness of the group $G$ (see Remark \ref{rem.prufer}(b), below). In \cite{PN}, Parvizi and Niroomand prove that if $G$ is a finitely generated subgroup and the non-abelian tensor square $[G,G^{\varphi}]$ is finite, then $G$ is finite. We obtain equivalence conditions (see the following theorem) and a topological related result (see Corollary \ref{cor.torsion}).

\begin{thmC} \label{thm.torsion}
Let $G$ be a finitely generated group. The following properties are equivalents. 
\begin{itemize}
\item[(a)] The group $G$ is finite;  
\item[(b)] The set of tensors $\tensor(G)$ is finite;
\item[(c)]  The non-abelian tensor $[G,G^{\varphi}]$ is finite; 
\item[(d)] The derived subgroup $G'$ is locally finite and the kernel $J_2(G) \cong \pi_3(SK(G,1))$ is periodic;
\item[(e)] The derived subgroup $G'$ is locally finite and the Diagonal subgroup $\Delta(G)$ is periodic;

\item[(f)] The derived subgroup $G'$ is locally finite and the  subgroup $\tilde{\Delta}(G) = \langle [g,h^{\varphi}][h,g^{\varphi}] \mid g,h \in G \rangle$ is periodic;

\item[(g)] The non-abelian tensor square $[G,G^{\varphi}]$ is locally finite.    
\end{itemize}
\end{thmC}

In algebraic topology, non-abelian tensor product arises from a homotopy pushout, see \cite{BL, BHS}. The  homotopy pushout (or homotopy amalgameted sum) is well known for the application in the classical Seifert-van Kampen theorem as well as Higher Homotopy Seifert-van Kampen theorem in the case of a covering by two open sets. Our contribution is to apply the idea of Theorem A and related construction on tensor product in order to obtain finiteness results related to homotopy pushout. For instance, see Proposition \ref{prop_GMN} for an application to homotopy pushout of Eilenberg-Maclane spaces.

The paper is organized as follows. In the next section we describe finiteness criteria for the group $\pi_n(X)$ in terms of the number of tensors. In particular, we establish a quantitative version of the classical Blakers-Massey triad connectivity theorem. In the third section we examine some finite necessary conditions for the group $G$ in terms of certain torsion elements of the non-abelian tensor square $[G,G^{\varphi}]$. In the final section, as an application we obtain finiteness criteria for the homotopy pushout that depends on the number of tensor of non-abelian tensor product of groups.  

\section{Finiteness Conditions}

Let $(X,A,B)$ be a triad, that is, $A$ and $B$ are subspaces of $X$, containing the base-point in $C=A\cap B$, such that the triad homotopy group $\pi_n(X,A,B)$ for $n\geq 3$ fit into a long exact sequence 
$$
\cdots \rightarrow \pi_n(B,C)\rightarrow \pi_n(X,A)\rightarrow \pi_n(X,A,B) \rightarrow \pi_{n-1}(B,C)\rightarrow\cdots.
$$
Let $X$ be a pointed space and $\{A, B\}$ an open cover of $X$ such that $A, B$ and $C=A\cap B$ are connected and $(A, C), (B, C)$ are 1-connected, see \cite{Wi}. 

By using the relative homotopy long exact sequences and the third triad homotopy group, Ellis and McDermott obtain interesting bound to the order of $\pi_3(X)$ (cf. \cite[Proposition 5]{EM}). We have obtained (as a consequence of Theorem A) a more general version. In contrast to their bound, we do not require that $\pi_2(A,C)$ and $\pi_2(B,C)$  be finite groups and also there is no need to estimate $\pi_2(X,C)$. 

Following the same setting, we apply the extended Blakers-Massey triad connectivity \cite[Theorem 4.2]{BL2}
$$\pi_{p+q+1}(X,A,B)\cong \pi_{p+1}(A,A\cap B)\otimes \pi_{q+1}(B,A\cap B),$$ in order to present finiteness condition and bound to $\pi_{n}(X)$ in terms of set of tensors  
$\tensor(\pi_{p+1}(A,C),\pi_{q+1}(B,C)),$ where $n=p+q+1$.

\begin{proof}[Proof of Theorem A]
Consider the relative homotopy long exact sequences, as in \cite{Wi},
$$\pi_n(B)\rightarrow \pi_n(B,C)\rightarrow \pi_{n-1}(C)$$
$$\pi_n(A)\rightarrow \pi_n(X)\rightarrow \pi_n(X,A)$$
$$\pi_n(B,C)\rightarrow \pi_n(X,A)\rightarrow \pi_n(X,A,B).$$

By the first exact sequence, we deduce that $\pi_n(B,C)$ is a finite group with $\{b,c\}$-bounded order. According to Brown-Loday's result \cite[Theorem 4.2]{BL2}, the group $\pi_n(X,A,B)$ is isomorphic to the non-abelian tensor product $M \otimes N$, where $M = \pi_{p+1}(A,C)$ and $N = \pi_{q+1}(B,C)$. As $|T_{\otimes}(\pi_{p+1}(A,C),\pi_{q+1}(B,C))|=m$, we have $\pi_n(X,A,B)$ is finite with $m$-bounded order (\cite[Theorem B]{BNRb}). From this we conclude that the group $\pi_n(X,A)$ is finite with $\{b,m\}$-bounded order. In the same manner we can see that $\pi_n(X)$ is finite with $\{a,b,c,m\}$-bounded order. The proof is complete.     
\end{proof}

\begin{rem}\label{rem_thmA}
A direct application of Theorem A is that 
$\pi_3(K(G,2)\vee K(H,2))$ is a finite group with $m$-bounded order, when $|\tensor (G,H)|=m$. In particular, 
$\pi_3(K(C_{r^\infty},2)\vee K(C_{s^\infty},2))$ is trivial, where $s$ and $r$ are primes and $K(G,n)$ is Eilenberg-MacLane space (a topological space having just one nontrivial homotopy group $\pi_n(K(G,n))\cong G$).
\end{rem}

By using the same idea of the previous remark we have the following corollary.

\begin{cor} \label{cor.bounded}
Let $A$ and $B$ be $p$-connected and $q$-connected locally contractible spaces, respectively. Suppose that $|\pi_n(A)|=a$, $|\pi_n(B)|=b$ and $|\tensor(\pi_{p+1}(A),\pi_{q+1}(B))|=m$, where $n=p+q+1$. Then $\pi_n(A\vee B)$ is finite with  $\{a,b,m\}$-bounded order.
\end{cor}

The previous corollary is an example that we can use the non-abelian tensor product to overcome the failure of the excision property, which reflects in $\pi_n(A\vee B)$ being different of $\pi_n(A)\oplus \pi_n(B)$ in general, where $n\geq 2$.

The following result provide a finiteness criterion to the group $G$ in terms of the number of tensors in the non-abelian tensor square $[G,G^{\varphi}]$.

\begin{cor} \label{cor.quantitative}
Let $X$ be connected space and $\pi_1(X)=G$. Suppose that the first homology group of $X$, $H_1(X,\mathbb{Z})$, is finitely generated and the set of tensors $\tensor(G) \subseteq \nu(G)$ has exactly $m$ tensors. Then $H_1(X,\mathbb{Z})$ and $\pi_1(X)=G$ are finite with $m$-bounded orders. 
\end{cor}

\begin{proof} Recall that $X$ is a connected space and $\pi_1(X)=G$. Suppose that the first homology of $X$, $H_1(X,\mathbb{Z})$, is finitely generated and the set of tensors $\tensor(G) \subseteq \nu(G)$ has exactly $m$ tensors. We need to prove that $H_1(X,\mathbb{Z})$ and $\pi_1(X)=G$ are finite with $m$-bounded order.  

By Theorem A, the non-abelian tensor square $[G,G^{\varphi}]$ is finite with $m$-bounded order. Consequently, the derived subgroup $G'$ is finite with $m$-bounded order. Since $G^{ab}$ is finitely generated, it follows that the abelianization $G^{ab}$ is isomorphic to an subgroup of the non-abelian tensor square $[G,G^{\varphi}]$ (\cite[Theorem C (a)]{BNRa}). Therefore the abelianization $G^{ab}$ is finite with $m$-bounded order. Consequently, $G$ is finite with $m$-bounded order. The proof is complete.  
\end{proof}

 In the case of the suspension triad $(SX;C_+ X, C_-X)$, see \cite{BL, Wi}, we can obtained finiteness criteria and bounds to the order of $\pi_3(SX)$, where $C_-X$ and $C_+X$ are the two cones of $X$ in $SX$.

\begin{proof}[Proof of Theorem B]
Consider the long exact sequence
$$\cdots\rightarrow\pi_2(X)\rightarrow \pi_3(SX)\rightarrow \pi_2(\Omega SX,X)\rightarrow \cdots$$
where $\Omega SX$ is the space of loops in $SX$, maps from the circle $S^1$ to $SX$, equipped with the compact--open topology.
The result follows by applying Theorem A on $ [G,G^{\varphi}] \cong \pi_2(\Omega SX,X)$.
\end{proof}

\begin{rem} \label{rem.suspension}
In the above result, it is worth noting that $\pi_2(X)$ does not need to be trivial, therefore the result is more general compared to the bounds for $\pi_3(SK(G,1))$ when $X=K(G,1)$. However, in \cite{BL}, proved that if $\pi_1(X) =G$ and $\pi_2(X)$ is trivial, then $\pi_3(SX)\cong J_2(G) = \ker(\kappa)$, where $\kappa: [G,G^{\varphi}] \to G'$, given by $[g,h^{\varphi}] \mapsto [g,h]$.
\end{rem}

Combining the above bounds to the order of the non-abelian tensor square and \cite[Proposition 4.10]{BL} we can obtained,  under appropriate conditions in the set of tensors $\tensor(G)$, some finiteness criteria and bounds to the orders of $\pi_3(SK(G,1))$ and $\pi_2^S(K(G,1))$ the second stable homotopy of Eilenberg-MacLane space.

\begin{cor}\label{cor_SKpi}
Let $G$ be a group. Suppose that the set $\tensor(G)$ has exactly $m$ tensors in $\nu(G)$. Then
\begin{itemize}
\item[(a)] The second stable homotopy group $\pi_2^{S}(K(G,1))$ is finite with $m$-bounded order; 
\item[(b)] $\pi_3(SK(G,1))$ is finite with $m$-bounded order. 
\end{itemize}
\end{cor}

A particular case of the previous corollary is when $G$ is a Pr\"ufer group $C_{p^\infty}$ so $\pi_2^{S}(K(G,1))$ and $\pi_3(SK(G,1))$ are trivial groups.

Next we give a sufficient and necessary conditions to the suspension of Eilenberg-Maclane space $K(G,1)$  be a finite group. 

\begin{prop}
Let G be a BFC-group such that $G^{ab}$ is finitely generated. Then, $\pi_3(SK(G,1))$ is finite if and only if $\pi_1(K(G,1))$ is finite.
\end{prop}

\begin{proof}
Assume that $J_2(G) \cong \pi_3(SK(G,1))$ is finite. We have the following short exact sequence $$0\rightarrow J_2(G) \rightarrow [G,G^{\varphi}] \rightarrow G'\rightarrow 1,$$ since $G'$ is finite by Neumann's Theorem \cite[14.5.11]{Rob}, we have that the non-abelian tensor square $[G,G^{\varphi}]$ is finite. Since $G^{ab}$ is finitely generated, it follows that the abelianization $G^{ab}$ is isomorphic to a subgroup of the non-abelian tensor square $[G,G^{\varphi}]$ (\cite[Theorem C (a)]{BNRa}). Moreover, we can deduce that the abelianization $G^{ab}$ is a subgroup of the diagonal subgroup $\Delta(G)$. From this we deduce that the abelianization $G^{ab}$ is finite and so, $G$ is finite.

Conversely, suppose that $\pi_1(K(G,1))$ is finite. Consequently, the non-abelian tensor square $[G,G^{\varphi}]$ is finite and so, $\pi_3(SK(G,1))$ is finite. The proof is complete. 
\end{proof}

\begin{rem} \label{rem.prufer}
(a) \ Assume that $G$ is a BFC-group and the abeliani\-zation $G^{ab}$ is finitely generated. Consider the following short exact sequence (cf. \cite[Section 2]{NR2}), 
$$
1 \to \Delta(G) \to J_2(G) \to H_2(G)\to 1,
$$
where $J_2(G)$ is isomorphic to $\pi_3(SK(G,1))$ and $H_2(G)$ is the second homology group. Recall that the Diagonal subgroup $\Delta(G)$ is given by $\Delta(G) = \langle [g,g^{\varphi}] \mid g \in G\rangle$. The subgroup $\Delta(G)$ is finite if and only if $\pi_1(K(G,1))$ is finite (if and only if the group $\pi_3(SK(G,1))$ is finite). See \cite[Section 2]{BL} for more details. \\

\noindent (b) Note that the hypothesis of the finitely generated abelianization is really needed. For instance, the Prufer group  $G=C_{p^{\infty}}$ is an infinite group such that the non-abelian tensor square  $[G,G^{\varphi}]$ is trivial and so, finite. In particular, $J_2(G) \cong \pi_3(SK(G,1))$ is also trivial.  
\end{rem}
	
A direct application of Corollary $\ref{cor.quantitative}$ and Corollary $\ref{cor_SKpi}$ for the suspension of Eilenberg-Maclane space $K(G,1)$ and the second stable homotopy of $K(G,1)$ is the following finiteness condition.
	
\begin{cor} Let $G$ be a group. Then, $\pi_3(SK(G,1))$ and $G'$ are finite if and only if $\tensor(G)$ is finite. Moreover, suppose that $G$ is perfect then, $\pi_2^S(K(G,1))$ and $G$ are finite if and only if $\tensor(G)$ is finite.
\end{cor}

\section{Torsion elements in the non-abelian tensor square}

This section is will be devoted to obtain some finiteness conditions for the group $G$ in terms of the torsion elements in the non-abelian tensor square. Specifically, our proofs involve looking at the description of the diagonal subgroup $\Delta(G) \leqslant [G,G^{\varphi}]$. Such a description has previously been used by the authors \cite{BNRa,NR2}. 

It is well known that the finiteness of the non-abelian tensor square $[G,G^{\varphi}]$ does not imply the finiteness of the group $G$ (see Remark \ref{rem.prufer}, above). In \cite{PN}, Parvizi and Niroomand prove that if $G$ is a finitely generated subgroup and the non-abelian tensor square $[G,G^{\varphi}]$ is finite, then $G$ is finite. Later, in \cite{BNRa}, the authors prove that if $G$ is a finitely generated locally graded group and the exponent of the non-abelian tensor square $\exp([G,G^{\varphi}])$ is finite, then $G$ is finite. The next result can be viewed as a generalization of the above result. 

\begin{thmC} \label{thm.torsion}
Let $G$ be a finitely generated group. The following properties are equivalents. 
\begin{itemize}
\item[(a)] The group $G$ is finite;  
\item[(b)] The set of tensors $\tensor(G)$ is finite;
\item[(c)]  The non-abelian tensor $[G,G^{\varphi}]$ is finite; 
\item[(d)] The derived subgroup $G'$ is locally finite and the kernel $J_2(G) \cong \pi_3(SK(G,1))$ is periodic;
\item[(e)] The derived subgroup $G'$ is locally finite and the Diagonal subgroup $\Delta(G)$ is periodic;

\item[(f)] The derived subgroup $G'$ is locally finite and the  subgroup $\tilde{\Delta}(G)$ is periodic;

\item[(g)] The non-abelian tensor square $[G,G^{\varphi}]$ is locally finite.    
\end{itemize}
\end{thmC}

\begin{proof}
$(a) \Rightarrow (b)$ and $(c) \Rightarrow (d) \Rightarrow (e) \Rightarrow (f)$ are directly. \\ 

$(b) \Rightarrow (c).$  Suppose that $\tensor(G)$. By \cite[Theorem A]{BNRa}, the non-abelian tensor square $[G,G^{\varphi}]$ is finite. \\ 

$(f) \Rightarrow (g)$. By Schmidt's Theorem \cite[14.3.1]{Rob}, it suffices to prove that the abelianization $G^{ab}$ is finite. 

Since $G^{ab}$ is finitely generated, we deduce that $$G^{ab} = T \times F,$$ where $T$ is the torsion part and $F$ the free part of $G^{ab}$ (cf. \cite[4.2.10]{Rob}). If the abelianization $G^{ab}$ is not periodic, then there exists an element of infinite order $x \in G$ such that $xG^{ab} \in F$. In particular, $[x,x^{\varphi}][x,x^{\varphi}] = [x,x^{\varphi}]^2$ is an infinite element in $\tilde{\Delta}(G)$. Consequently, $F$ is trivial and $G^{ab} = T$ is finite. \\

$(g) \Rightarrow (a).$ First we prove that the abelianization $G^{ab}$ is finite. Arguing as in the above paragraph, we deduce that $G^{ab} = T \times F,$ where $T$ is the torsion part and $F$ the free part of $G^{ab}$. From \cite[Remark 5]{NR2} we conclude that  $\Delta(G^{ab})$ is isomorphic to $$\Delta(T) \times \Delta(F) \times (T \otimes_{\mathbb{Z}} F),$$ where $T \otimes_{\mathbb{Z}} F$ is the usual tensor product of $\mathbb{Z}$-modules. In particular, the free part of $\Delta(G^{ab})$ is precisely $\Delta(F)$. Now, the canonical projection $G \twoheadrightarrow{G^{ab}}$ induces an epimorphism $q: \Delta(G) \twoheadrightarrow {\Delta(G^{ab})}$. Since $\Delta(G)$ is locally finite, it follows that $\Delta(G^{ab})$ is also locally finite. Consequently, $F$ is trivial and thus $G^{ab}$ is periodic and, consequently, finite.

It remains to prove that the derived subgroup $G'$ is finite. Since $G$ is finitely generated and $G^{ab}$ is finite, it follows that the derived subgroup $G'$ is finitely generated (\cite[1.6.11]{Rob}). As $G'$ is an homomorphic image of the non-abelian tensor square $[G,G^{\varphi}]$, we have $G'$ is finite. From this we deduce that $G$ is finite. The proof is complete.  
\end{proof}

\begin{prop}\label{prop_poly}
Let $G$ be a polycyclic-by-finite group. Suppose that the non-abelian tensor $[G,G^{\varphi}]$ is periodic. Then $G$ is finite.   
\end{prop}
\begin{proof}
Since the derived subgroup $G'$ is an epimorphic image of the non-abelian tensor square $[G,G^{\varphi}]$, it follows that $G'$ is also periodic. In particular, we can deduce that $G'$ is finite. Now, arguing as in the proof of Theorem \ref{thm.torsion}, we deduce that $G^{ab}$ is finite. The proof is complete. 
\end{proof}

It is well known that if $G$ is a group with exponent $\exp(G)\in \{2,3,4,6\}$, then $G$ is locally finite (Levi-van der Waerden, Sanov, see \cite[Section 14.2]{Rob} for more details). We also examine the finiteness of the group $G$, when the non-abelian tensor square $[G,G^{\varphi}]$ has small exponent. 

\begin{cor}\label{cor_Burnside}
Let $n \in \{2,3,4,6\}$ and $G$ a finitely generated group. Assume that the exponent of the non-abelian tensor square $\exp([G,G^{\varphi}])$ is exactly $n$. Then $G$ is finite.  
\end{cor}

\begin{proof}
It suffices to see that the non-abelian tensor square $[G,G^{\varphi}]$ is locally finite (see \cite[Section 14.2]{Rob} for more details). By Theorem \ref{thm.torsion}, the group $G$ is finite. The proof is complete. 
\end{proof}

The following corollary is a topological version of Theorem C.

\begin{cor}\label{cor.torsion}
Let $G$ be a finitely generated group and $X$ be a topological space such that $\pi_1(X)=G$ and $\pi_2(X)$ is trivial. Then the following properties are equivalents.

\begin{itemize}
    \item[(a)] The group $\pi_1(X)=G$ is finite;  
\item[(b)] The derived subgroup $G'$ is locally finite and $\pi_3(SX)$ is periodic;
\item[(c)] The derived subgroup $G'$ is locally finite and  $ F(\pi_3(SX)) \cong F(H_2(G))$, where $F(H)$ is the free part of a group $H$;

\item[(d)] The derived subgroup $G'$ is locally finite and $F(\pi_3(SX)) \cong F(\pi_4(S^2X))$, where $F(H)$ is the free part of a group $H$.
\end{itemize}
\end{cor}

\begin{proof} Since $\pi_2(X)$ is trivial, it follows that $J_2(G)\cong \pi_3(SX)$ (cf. \cite[Proposition 3.3]{BL}). Consider the short exact sequences, as in \cite{BL}, $$1 \rightarrow J_2(G) \rightarrow [G,G^{\varphi}] \rightarrow G'\rightarrow 1,$$
$$1 \rightarrow \tilde{\Delta}(G) \rightarrow J_2(G) \rightarrow \pi_4(S^2X)\rightarrow 1.$$

Therefore, the result follows by applying Theorem C.
\end{proof}

\section{Application to homotopy pushout}

We end this paper by proving some finiteness criteria for the homotopy pushout. Consider the following commutative square of spaces 	$$\xymatrix{ C \ar[d]^g \ar[r]^f & A \ar[d]^a \\ B \ar[r]^b  & X 
	}$$ 
and denote $F(f), F(g)$ and $F(a)$ the homotopy fibre of $f, g$ and $a$, respectively, and let $F(X)$ be the homotopy fibre of $F(g)\rightarrow F(a)$. The previous square is called a {\it homotopy pushout} when the canonical map of squares from the double mapping cylinder $\mathrm{M}(f,g)$ to $X$ is a weak equivalence of spaces at the four corners, for more details see \cite{BL}. For the homotopy pushout, we have an analogous of Theorem A. 

\begin{prop}\label{homotopy_sum}
Let the following square of spaces be a homotopy pushout $$\xymatrix{ C \ar[d]^g \ar[r]^f & A \ar[d]^a \\ B \ar[r]^b  & X 
}$$
Suppose that  $\pi_3(A)$, $\pi_2(F(g))$ and the set of tensors $\tensor(G,H)$ are finite, where $G = \pi_1(F(f))$ and $H = \pi_1(F(g))$. Then $\pi_3(X)$ is a finite group with $\{n_a,n_b,m\}$-bounded order, where $|\pi_3(A)|=n_a$, $|\pi_2F(g)|=n_b$ and $|\tensor(G,H)|=m$.    
\end{prop}
\begin{proof}
By \cite[Theorem 3.1]{BL} we have that $\pi_1(F(X))\simeq \pi_1(F(f))\otimes \pi_1(F(g))$. Since $|T_{\otimes}(\pi_1(F(f)),\pi_1(F(g)))|=m$ then  $\pi_1(F(X))$ is finite with $m$-bounded order (\cite[Theorem B]{BNRb}). By using $|\pi_3(A)|=n_a$, $|\pi_2F(g)|=n_b$ and the long exact sequences of the fibrations $$F(X)\rightarrow F(g)\rightarrow F(a)$$ $$F(a)\rightarrow A\rightarrow X,$$ it follows that  $\pi_3(X)$ is a finite group with $\{n_a,n_b,m\}$-bounded.
\end{proof}

Many authors had studied some finiteness conditions for the non-abelian tensor product of groups (cf. \cite{BNRa,BR1,DLT,Ellis,M,PN}). For instance, in \cite{M}, Moravec proved that if $G,H$ are locally finite groups acting compatibly on each other, then so is $G \otimes H$. In \cite{DLT}, Donadze, Ladra and Thomas proved interesting finiteness criteria for non-abelian tensor product in terms of the involved groups. In \cite{BNRa,BNRb}, the authors prove a finiteness criterion for the non-abelian tensor product of groups in terms of the number of tensors. In the case of homotopy pushout for Eilenberg-Maclane spaces is possible to obtain results direct from the study of the non-abelian tensor product of groups. We obtain the following related result. 

\begin{prop}\label{prop_GMN} Let $M,N$ be normal subgroups of a group $G$, and form the homotopy pushout
	
$$\xymatrix{ K(G,1) \ar[d] \ar[r] & K(G/N,1) \ar[d] \\ K(G/M,1) \ar[r]  & X 
}$$
\begin{itemize}
\item[(a)] Suppose that  the subgroups $M$ and $N$ are locally finite. Then $\pi_2(X)$ and $\pi_3(X)$ are locally finite.  
\item[(b)] Suppose that $M$ is a non-abelian free group of finite rank and $N$ is a finite group. Then $\pi_2(X)$ and $\pi_3(X)$ are finite.  
\item[(c)] 
Suppose that the set of tensors $\tensor(M,N)  \subseteq \eta(M,N)$ is finite. Then $\pi_3(X)$ is a finite group with $m$-bounded order, where $|\tensor(M,N)|=m$.  
\end{itemize}	
\end{prop}
\begin{proof}

According to Brown and Loday's result \cite[Corollary 3.4]{BL}, the group $\pi_3(X)$ is isomorphic to a subgroup of the non-abelian tensor product $[M,N^{\varphi}]$ and 
$\pi_2(X)$ is isomorphic to $(M \cap N)/[M,N]$. In particular, the group $\pi_2(X)$ is an homomorphic image of $M\cap N$. \\

\noindent (a). As $M \cap N$ is locally finite we have $\pi_2(X)$ is locally finite. Now, since $M$ and $N$ are locally finite, it follows that the non-abelian tensor product $[M,N^{\varphi}]$ is locally finite (Moravec, \cite{M}). Consequently, the group $\pi_3(X)$ is locally finite.  \\

\noindent (b). Since $M$ is finite and $N$ is a non-abelian free group, we deduce that $\pi_2(X)$ is finite. According to Donadze, Ladra and Thomas' result \cite[Corollary 4.7]{DLT}, we conclude that the non-abelian tensor product $[M,N^{\varphi}]$ is finite and so, $\pi_3(X)$ is finite. \\

\noindent (c). According to Theorem A, the non-abelian tensor product $[M,N^{\varphi}]$ is finite with $m$-bounded order. In particular, the group $\pi_3(X)$ is finite with $m$-bounded order. The proof is complete.   
\end{proof}

\begin{rem}
Note that Proposition \ref{prop_GMN} (c) in a certain sense cannot be improved. For instance, if $M=N=C_{p^{\infty}}$, then the group $\pi_2(X)\cong C_{p^{\infty}}$ is infinite and $\pi_3(X)$ is trivial. 
\end{rem}

As an interesting consequence of 
Proposition \ref{prop_GMN}, we obtain that: if
$G=MN$ such that $M\cap N$ and $[M,N^{\varphi}]$ are trivial, then $X$ is 3-connected, i.e., $\pi_n(X)$ is trivial for $n=1,2$ and $3$, see the following example.

\begin{ex}
When $G=C_{r^\infty} \times C_{s^\infty}$ in Proposition \ref{prop_GMN} with $r$ and $s$ primes (not necessarily distinct), we have that $X$ is $3$-connected. In fact, $\pi_1(X)$ is trivial by Van Kampen theorem, as well $\pi_2(X)=0$ as a consequence of the Van Kampen theorem for maps, and finally the triviality of $\pi_3(X)$ follows from the proposition above. 
\end{ex}


\begin{thebibliography}{10}
\bibitem{ADPT} A. Antony, G. Donadze, V. Prasad and V. Z. Thomas, {\it The second stable homotopy group of the Eilenberg-Maclane space}, Math. Z., {\bf 287} (2017) pp. 1327--1342. 

\bibitem{BNRa} R. Bastos, I.\,N. Nakaoka and N.\,R. Rocco, {\it Finiteness conditions for the non-abelian tensor product of groups}, Monatsh. Math., {\bf 187} (2018) pp. 603--615.    

\bibitem{BNRb} R. Bastos, I.\,N. Nakaoka and N.\,R. Rocco, {\it The order of the  non-abelian tensor product of groups},  arXiv:1812.04747v1 [math.GR]. 

\bibitem{BR1} R. Bastos and N.\,R. Rocco, {\it The non-abelian tensor square of residually finite groups}, Monatsh. Math., {\bf 183} (2017) pp. 61--69.  

\bibitem{BR2} R. Bastos and N.\,R. Rocco, {\it Non-abelian tensor product of residually finite groups}, S\~ao Paulo J. Math. Sci., {\bf 11} (2017) pp. 361--369. 

\bibitem{BK} J.\,R. Beuerle and L.-C. Kappe, {\it Infinite metacyclic groups and their non-abelian tensor squares}, Proc. Edinb. Math. Soc., {\bf 43} (2000) pp. 651--662. 
	
\bibitem{BJR} R. Brown, D.\,L. Johnson and E.\,F. Robertson, {\it Some computations of non-abelian tensor products of groups}, J. Algebra, {\bf 111}  (1987)  pp. 177--202.

\bibitem{BL} R. Brown, and J.-L. Loday, {\it Van Kampen theorems for diagrams of spaces}, Topology, {\bf 26} (1987) pp. 311--335.

\bibitem{BL2} R. Brown, and J.-L. Loday, {\it Homotopical excision, and Hurewicz theorem, for $n$-cubes of spaces}, Proc. London Math. Soc.(3), {\bf 4} (1987) pp. 176--192.

\bibitem{BHS} R. Brown, P. Higgins and R. Sivera, {\it Nonabelian algebraic topology. Filtered spaces, crossed complexes, cubical homotopy groupoids.} With contributions by Christopher D. Wensley and Sergei V. Soloviev. EMS Tracts in Mathematics, 15. European Mathematical Society (EMS), Zürich, 2011. xxxvi+668 pp. ISBN: 978-3-03719-083-8.

\bibitem{DLT} G. Donadze, M. Ladra and V. Thomas, {\it On some closure properties of the non-abelian tensor product}, J. Algebra, {\bf 472} (2017) pp. 399--413. 

\bibitem{Ellis} G. Ellis, {\it The non-abelian tensor product of finite groups is finite}, J. Algebra, {\bf 111} (1987) pp. 203--205.

\bibitem{EL} G. Ellis and F. Leonard, {\it Computing Schur multipliers and tensor products of finite groups},  Proc. Royal Irish Acad., {\bf 95A} (1995) pp. 137--147.

\bibitem{EM} G. Ellis and A. McDermott, {\it Tensor products of prime-power groups}, J. Pure Appl. Algebra, {\bf 132} (1998) pp. 119--128. 

\bibitem{M} P. Moravec, {\it The exponents of nonabelian tensor products of groups}, J. Pure Appl. Algebra, {\bf 212} (2008)  pp. 1840--1848.

\bibitem{Nak} I.\,N. Nakaoka,  {\it Non-abelian tensor products of solvable groups}, J. Group Theory, {\bf 3} (2000) pp. 157--167.

\bibitem{NR14} P. Niroomand and F.\,G. Russo, {\it On the size of the third homotopy group of the suspension of an Eilenberg-MacLane space}, Turk. J. Math., {\bf 38} (2014) pp. 664--671.

\bibitem{PN} M. Parvizi and P. Niromand, {\it On the structure of groups whose exterior or tensor square is a $p$-group}, J. Algebra, {\bf 352} (2012) pp. 347--353.

\bibitem{Rob} D.\,J.\,S. Robinson,
\textit{A course in the theory of groups}, 2nd edition, Springer-Verlag, New York, 1996.

\bibitem{NR1} N.\,R. Rocco, {\it On a construction related to the non-abelian tensor square of a group}, Bol. Soc. Brasil Mat., {\bf 22} (1991) pp. 63--79.

\bibitem{NR2} N.\,R. Rocco, {\it A presentation for a crossed embedding of finite solvable groups}, Comm. Algebra {\bf 22} (1994) pp. 1975--1998.

\bibitem{Wi} J.H.C. Whitehead, {\it A certain exact sequence}, Ann. Math. {\bf 52} (1950) pp. 51--110.
\end{thebibliography}
\end{document}